\documentclass[11pt]{amsart}
 
\theoremstyle{plain}
\newtheorem{thm}{Theorem}[section]
\newtheorem{theorem}[thm]{Theorem}

\newtheorem{corollary}[thm]{Corollary}
\newtheorem{proposition}[thm]{Proposition}

\newtheorem{remark}[thm]{Remark}
\newtheorem{remarks}[thm]{Remarks}

\numberwithin{equation}{section}

\title{Some observations on the properness of Identity plus linear powers} 
\author{Tuyen Trung Truong}
\subjclass[2010]{13B25, 14Qxx, 14Rxx, 15Bxx}
\keywords{Druzkowski maps; Groebner basis; Jacobian conjecture; Properness}
\date{\today}
\address{University of Oslo, 0316 Oslo, Norway }\email{tuyentt@math.uio.no}
\begin{document}

\begin{abstract}
\noindent 

 For $2$ vectors $x,y\in \mathbb{R}^m$, we use the notation $x * y =(x_1y_1,\ldots ,x_my_m)$, and if $x=y$ we also use the notation $x^2=x*x$ and define by induction $x^k=x*(x^{k-1})$. We use $<,>$ for the usual inner product on $\mathbb{R}^m$. For $A$ an $m\times m$ matrix with coefficients in $\mathbb{R}$, we can assign a map $F_A(x)=x+(Ax)^3:~\mathbb{R}^m\rightarrow \mathbb{R}^m$. A matrix $A$ is Druzkowski iff $det(JF_A(x))=1$ for all $x\in \mathbb{R}^m$.   

Recently, Jiang Liu posted a preprint on arXiv asserting a proof of the Jacobian conjecture, by showing the properness of $F_A(x)$ when $A$ is Druzkowski, via some inequalities in the real numbers. In the proof, indeed Liu asserted the properness of $F_A(x)$ under more general conditions on $A$, see the main body of this paper for more detail. 

Inspired by this preprint, we research in this paper on the question of to what extent the above maps $F_A(x)$ (even for matrices $A$ which are not Druzkowski) can be proper. We obtain various necessary conditions and sufficient conditions for both properness and non-properness properties. A complete characterisation of the properness, in terms of the existence of  non-zero solutions to a system of polynomial equations of degree at most $3$, in the case where $A$ has corank $1$, is obtained. Extending this, we  propose a new conjecture, and discuss some applications  to the (real) Jacobian conjecture. We also consider the properness of more general maps $x\pm (Ax)^k$ or $x\pm A(x^k)$. By a result of Druzkowski, our results can be applied to all polynomial self-mappings of $\mathbb{C}^m$ or $\mathbb{R}^m$.  
\end{abstract}

\maketitle


\section{Introduction and results}

In this paper we discuss the properness of a special class of maps defined on the field $\mathbb{R}$, inspired by the Jacobian Conjecture. Interestingly, the fact that $\mathbb{R}$ is well-ordered, together with the fact that every real number has a unique real third-root, gives an apparent advantage of considering the properness on the field $\mathbb{R}$ over on the field $\mathbb{C}$.  

The plan of this paper is as follows. In the first subsection we recall known facts about the Jacobian conjecture, as well as some reductions to Druzkowski maps and injectivity or properness of maps. In the second and third subsections, we provide various necessary conditions and sufficient conditions for the properness of maps of the same type as Druzkowski maps, but applied to a general matrix $A$ (need not be a Druzkowski matrix). We then propose a new conjecture, which asserts that the non-properness of $F_A(x)$ is equivalent to that a system of polynomial equations (of degree at most $3$) has a non-zero solution. The number of equations and variables in the system are independent of the matrix $A$. In the fourth subsection we discuss some applications  to the Jacobian conjecture. The fifth subsection extends the results to even more general maps. Because of a result by Druzkowski \cite{druzkowski}, which states that all self-polynomials on $\mathbb{C}^m$ or $\mathbb{R}^m$ can be reduced to a map of the form considered in this paper, our results can be applied to all polynomial self-maps with real or complex coefficients. 

\subsection{Jacobian conjecture, Druzkowski matrices, injectivity and properness}

We recall that the famous Jacobian Conjecture is the following statement:

\noindent
{\bf Jacobian Conjecture.} Let $F=(F_1,\ldots ,F_m):\mathbb{C}^m\rightarrow \mathbb{C}^m$ be a polynomial map such that $JF$ (the Jacobian matrix $(\partial F_i/\partial x_j)_{1\leq i,j\leq m}$) is invertible at every point. Then, $F$ has a polynomial inverse. 

The Jacobian conjecture was first stated by Keller in 1939. Polynomial maps with invertible Jacobian matrices are called Keller maps. We denote by $JC(m)$ the Jacobian Conjecture in dimension $m$, and by $JC(\infty )$ the statement that $JC(m)$ holds for every $n$. In the literature $JC(\infty )$ is usually called the generalized Jacobian Conjecture. This conjecture has attracted a lot of works, and many partial results were found. For example,  Magnus - Applegate -Onishi -Nagata proved $JC(2)$ for $F=(P,Q)$ where the GCD of the degrees of $P,Q$ is either a prime number or $\leq 8$;  Moh proved $JC(2)$  for $\deg (F)\leq 100$; Wang proved that $JC(n)$ holds if $\deg (F)=2$; and Yu \cite{yu} (also Chau and Nga \cite{chau-nga}) proved that if $F(X)-X$ has no linear term and has all non-positive coefficients then JC holds for $F$. For more details the readers can consult the reference list and the references therein. An excellent survey is the book \cite{essen}. We note that the $\mathbb{R}$-analog of $JC(2)$ (in this case, we require only that the map $F$ is bijective, since its inverse may not be a polynomial as the example $F(x)=x+x^3$ shows) is false, by the work of Pinchuk (Section 10 in \cite{essen}). 

There have been many reductions of the generalised Jacobian Conjecture $JC(\infty )$. One of these reductions is due to Bass, Connell, Wright and Yagzhev, who showed that to prove $JC(\infty )$, it is enough to prove for all   $F(x)=x+H(x)$ and all $n$, where $H(x)$ is a homogeneous polynomial of degree $3$ (Section 6.3 in \cite{essen}). Druzkowski made a further simplification (Section 6.3 in \cite{essen})  

\begin{theorem}[Druzkowski]
$JC(\infty )$ is true if it is true for all the maps $F$ of the form $F(x)=(x_1+l_1(x)^3,\ldots ,x_m+l_m(x)^3)$ with invertible Jacobian $JF$, here $l_1,\ldots ,l_m$ are linear forms. 
\label{TheoremDruzkowski}\end{theorem}
Later, Druzkowski \cite{druzkowski} simplified even further showing that it is enough to show for the above maps with the additional condition that $A^2=0$, where $A$ is the $n\times n$ matrix whose $i$-th row is $l_i$. We then simply say that a matrix is Druzkowski if the corresponding map $F_A(x)=(x_1+l_1(x)^3,\ldots ,x_m+l_m(x)^3)$ is Keller, i.e. the determinant of its Jacobian is $1$.

A lot of efforts have been devoted to showing that the Druzkowski maps are polynomial automorphisms (Section 7.1 in \cite{essen} and for recent developments see  \cite{bondt-yan1,bondt-yan2,bondt-yan3}, and for a comprehensive reference on this topic see \cite{bondt}).  There are many partial results proved for this class of maps, for example it is known from the works of Druzkowski, Hubbers, Yan and many other people that the Druzkowski maps are polynomial automorphisms if either the rank of $A$ is $\leq 4$ or the corank of $A$ is $\leq 3$. In particular, the Jacobian conjecture was completely checked for Druzkowski maps in dimensions $\leq 8$ (see \cite{bondt}, also for stronger properties that can be proved for these polynomial automorphisms). Some new results on Druzkowski maps in dimension $9$ have been obtained recently in \cite{yan} and \cite{bondt-yan1,bondt-yan2,bondt-yan3}. A common theme of these proofs is that in these cases the Druzkowski maps are "equivalent" to some other polynomial maps for which the Jacobian conjecture is previously known to be true. It is not easy to see whether this strategy can work for higher ranks or coranks.

We now present a compact form of maps including Druzkowski maps. Let $\mathbb{K}=\mathbb{R}$ or $\mathbb{C}$. For $2$ vectors $x\in \mathbb{K}^m$, we use the notation $x * y =(x_1y_1,\ldots ,x_my_m)$, and if $x=y$ we also use the notation $x^2=x*x$ and define by induction $x^k=x*(x^{k-1})$. If all coordinates of $x$ are non-zero, then we can define $x^{-k}$ to be the vector such that $(x^{-k})*(x^k)=(1,\ldots ,1)$. We use $<,>$ for the usual inner product on $\mathbb{K}^m$. For $A$ an $m\times m$ matrix with coefficients in $\mathbb{R}$, we can assign a map $F_A(x)=x+(Ax)^3:~\mathbb{K}^m\rightarrow \mathbb{K}^m$. We call these maps $F_A$ to be of type identity plus cubic. A matrix $A$ is Druzkowski iff $det(JF_A(x))=1$ for all $x\in \mathbb{R}^m$. 
  
For a Druzkowski map on $\mathbb{C}$, it is known (see \cite{essen}) that $F_A(x)$ is an automorphism iff it is injective. In \cite{truong}, the author showed that a map $F_A(x)$ (for a general matrix $A$ need not be Druzkowski)  is injective iff the equations $Z+A(Z^3)+A(Z*(AY)^2)=0$ only have solutions $Z=0$. This turns out to provide a very quick method to check whether a Druzkowski map is injective by computers. By using that for every triple $(A,Y,Z)$ satisfying the equation $Z+A(Z^3)+A(Z*(AY)^2)=0$, another triple $(A',Y',Z')$ with $A'=DAD^{-3}, Y'=D^3Y$ and $Z'=D^3Z$, with $D$ an invertible diagonal matrix, also satisfies the same equation, we can reduce the truth of $JC(m)$ to the following: 

{\bf Conjecture A:} For $Z=(1,\ldots ,1,0,\ldots ,0)$ (with $1\leq k\leq m$ entries being $1$), there is no Druzkowski matrix $A$ and vector $Y$ so that $Z+A(Z^3)+A(Z*(AY)^2)=0$.

We also proposed some generalisations of the Jacobian conjecture in \cite{truong}, applying to all matrices and not just Druzkowski matrices. We showed that these new conjectures are satisfied for a generic matrix of a given rank. 

Recently, Liu posted a preprint \cite{liu} asserting a proof of Jacobian conjecture. He proceeded to proving the following 3 assertions: (i) The Jacobian conjecture can be reduced to proving that: for a Druzkowski matrix $A$ with real coefficients, the map  $F_A(x):~\mathbb{R}^m\rightarrow \mathbb{R}^m$ is injective; (ii) By Hadamard's theorem, the injectivity of the map  $F_A(x):~\mathbb{R}^m\rightarrow \mathbb{R}^m$ is reduced to its properness; (iii) The properness of  the map  $F_A(x):~\mathbb{R}^m\rightarrow \mathbb{R}^m$ can be proven by using some simple inequalities on real numbers. 

Liu's argument for (i) is as follows. A Druzkowski map $F_A(x)$ on $\mathbb{C}$ is an automorphism iff it is injective. Now, if we write in terms of real and imaginary parts (note that the real and imaginary parts of $F_A(x)$ are Yaghzev's maps, with real coefficients, and hence can be transformed to a Druzkowski map with real coefficients), we reduce checking the injectivity of a Druzkowski map on $\mathbb{C}^m$ to the injectivity of a Druzkowski map on $\mathbb{R}^{2m}$.  

We recall that a map $F$ is proper iff the preimage of compact sets are compact. Equivalently, a map $F$ is proper if for every unbounded sequence $\{x_n\}$ we have $\{F(x_n)\}$ is unbounded. Hardamard's theorem, mentioned above, is the following hard-to-believe result (see \cite{gordon, nijenhuis-richardson}): 

\begin{theorem} Let $f:\mathbb{R}^m\rightarrow \mathbb{R}^m$ be a $C^1$ map so that $det (Jf)$ never vanishes. Then $f$ is a $C^1$-diffeomorphism iff $f$ is proper.  
\label{TheoremHadamard}\end{theorem}     
 
Liu proved the properness of maps $F_A(x)$ under the following assumption (which is satisfied in case $A$ is a Druzkowski matrix): $A$ is a matrix so that for all $\lambda \in \mathbb{R}$, the equation $F_{\lambda A}(x)=0$ has only one solution $x=0$.  Here, he uses some inequalities on the real numbers, however since the arguments are a bit complicated and un-natural, we still cannot yet confirm whether this last step in \cite{liu} is correct.    

\begin{remarks}It can be easily checked that a matrix $A$ satisfies the property used by Liu iff for $\lambda =\pm 1$, the equation $F_{\lambda A}(x)=0$ has only one solution $x=0$. We note that the set $\mathcal{Z}=$ $\{A:$ $F_A(x)=0$ has only the unique solution $x=0\}$ is constructible, since whose complement is a projection of the real algebraic variety $\{(A,x,s):$ $F_A(x)=0$ and $s||x||^2-1=0\}$. Hence, the set of matrices considered by Liu, which is exactly $\{A:~A,~-A\in \mathcal{Z}\}$, is constructible.  

The idea of showing properness of maps satisfying the Jacobian condition has been used before, for example in \cite{chamberland-meisters} where the idea of Mountain-pass theorem has been used. 
\end{remarks}

A general reader may be reluctant to dive deeply into the preprint \cite{liu}, given the bad reputation of Jacobian conjecture having many false claimed proofs. The fact that Liu proves the properness of $F_A(x)$ under more general assumptions than being Druzkowski could just add more suspect. The purpose of this paper is to convince the readers that the fact $F_A(x)$ is proper may not be too special. While we cannot confirm whether $F_A(x)$ is proper when $A$ is Druzkowski, we do show that $F_A(x)$ is proper for almost all matrices - for example, this is valid when $A$ is invertible or symmetric or anti-symmetric. For small dimensions $m=1,2$, actually $F_A(x)$ is proper for all matrices $A$. Indeed, the readers can readily check with the case $m=1$. The case $m=2$ is more difficult, and is a special case of Theorem \ref{Theorem1}. Moreover, we provide various necessary conditions and sufficient conditions, in terms of the existence of non-zero solutions to certain systems of polynomial equations, for $x+A(x^3)$ to be non-proper. A complete characterisation of the properness, in  the case where $A$ has corank $1$, is obtained. Extending this, we state Conjecture C. We will also consider properness of maps of the form $x\pm (Ax)^{k}$ or $x\pm A(x^{k})$. We hope that these results could encourage readers to check in detail whether \cite{liu} is correct.  

Before stating and proving these results, we present some preliminary results. For an $m\times m$ matrix $A$ with real coefficients, which can be thought of as a linear map $A:~\mathbb{R}^m\rightarrow \mathbb{R}^m$, we denote by $Ker(A)=\{x:~Ax=0\}$ the kernel of $A$ and $Im(A)=\{z:~z\in A(\mathbb{R}^m)\}$ the image of $A$. We denote by $A^T$ the transpose of $A$. We have the following Kernel-Image theorem in Linear Algebra: $Im(A^T)$ and $Ker(A)$ are orthogonal with respect to the inner product $<,>$, and $\mathbb{R}^m=Ker(A)\oplus Im(A^T)$. As a consequence, if $z\in Im(A^T)$ is non-zero, then $Az\not=0$. [To check this last claim, one can simply write $z=A^Tu$ and $<Az,u>=<z,A^Tu>=<z,z>$.] It seems that \cite{liu} is the first paper where the Kernel-Image theorem has been cleverly applied to the Jacobian conjecture.  

\subsection{Properness of Identity plus linear cubics} The next result gives useful criterions to checking the properness of $F_A(x)$. 

\begin{proposition} Let $A$ be an $m\times m$ matrix with real coefficients and $F_A(x)=x+(Ax)^3: ~\mathbb{R}^m\rightarrow \mathbb{R}^m$. 

1)  If $F_A(x)$ is not proper, then there is $x$ so that $A.(Ax)^3=0$ but $Ax\not=0$. Moreover, we can choose such an $x$ in $Im(A^T)$. 

2) $F_A(x)$ is proper iff  $F_B(x)$ is a proper, where $B=DAD^{-3}$ for some invertible diagonal matrix $D$. 

3) $F_A(x)$ is proper iff the map $\widehat{F}_A(x):=x+A(x^3)$ is proper on $Im(AA^T)$. 

\label{MainTheorem}\end{proposition}
\begin{proof}

1) Assume that $F_A(x)$ is not proper. Then, there is a sequence $\{z_n\}$ so that $||z_n||\rightarrow\infty$ but $F_A(z_n)$ is bounded. We write as before $z_n=x_n+y_n$, where $x_n\in Ker(A)$ and $y_n\in Im(A^T)$. Then, from $F_A(z_n)=x_n+y_n+(Ay_n)^3$ and the assumption that $F_A(z_n)$ is bounded and $\{z_n\}$ is unbounded, we obtain that $||y_n||\rightarrow\infty$. We can assume that $\lim _{n\rightarrow\infty}y_n/||y_n||=y_{\infty}$ exists and is non-zero. Therefore $A(y_{\infty})\not=0 $  since $0\not= y_{\infty}\in Im(A^T)$. Thus (using that $A(x_n)=0$ for all $n$)
\begin{eqnarray*}
0=\lim _{n\rightarrow\infty} A(\frac{F_A(z_n)}{||y_n||}) = A(y_{\infty})+\lim_{n\rightarrow\infty}||y_n||^2A((Ay_n/||y_n||)^3). 
\end{eqnarray*}

Since $\lim _{n\rightarrow\infty}||y_n||=\infty$ and $\lim _{n\rightarrow\infty}Ay_n/||y_n||=Ay_{\infty}$, it follows that $A((Ay_{\infty})^3)=0$.

2)  We show that $F_A$ is proper  if and only if $F_B$ is proper, where $B=DAD^{-3}$ for an invertible diagonal matrix $D$. Indeed, for every $x\in \mathbb{R}^m$ and an invertible diagonal matrix $D$, we define $z=D^3x$. Then, from the fact that $D$ is {\bf diagonal}, we have as in \cite{truong}:
\begin{eqnarray*}
B((Bz)^3)=DAD^{-3}((DAx)^3)=DA((Ax)^3).
\end{eqnarray*}
It is also easy to check that $Bz=DAx$. Therefore, $F_B(z)=DF_A(x)$. Since $D$ is invertible, it follows that if $z_n=D^3x_n$, then $||x_n||\rightarrow \infty$ iff $||z_n||\rightarrow \infty$, and $||F_A(x_n)||\rightarrow \infty$ iff $||F_B(z_n)||\rightarrow \infty$. Therefore, $F_A(x)$ is proper iff $F_B(x)$ is proper. 

3) We recall the big-O notation: if $\{v_n\}$ is a sequence of vectors and $\{r_n\}$ is a sequence of positive numbers, then we write $v_n=O(r_n)$ if there is a constant $C>0$ so that $||v_n||\leq Cr_n$ for all $n$. 

We first show the following: The map $F_A(x)$ is proper iff there is no sequence $y_n\in Im(A^T)$ so that $||y_n||\rightarrow \infty$ and $A(y_n+(Ay_n)^3) =O(1)$.

Proof of $\Rightarrow$: This is exactly the proof of part 1).

Proof of $\Leftarrow$: Assume that there is a sequence $y_n\in Im(A^T)$ so that $||y_n||\rightarrow 0$ and $A(y_n+(Ay_n)^3) =O(1)$. By linear algebra, there is $v_n$ so that $A(v_n)=A(y_n+(Ay_n)^3)$ and $||v_n||=O(||A(y_n+(Ay_n)^3)||)=O(1)$. (We recall how to prove this simple Linear Algebra fact. We can, by using changes of coordinates, bring $A$ to the normal Jordan form. Then, we can assume further that $A$ is a Jordan block. In this case, the conclusion is easy to check.) We define $x_n=v_n-(y_n+(Ay_n)^3)$. Then $x_n\in Ker(A)$. We define $z_n=x_n+y_n$. Since $||z_n||^2=||x_n||^2+||y_n||^2\geq ||y_n||^2$, it follows that $||z_n||\rightarrow \infty$. Also, $z_n+(Az_n)^3=x_n+y_n+(Ay_n)^3=v_n$ is bounded. Thus $F_A(x)$ is not proper.  

Now we finish the proof of part 3). We note that there are constants  $C_1,C_2>0$ such that $C_1||y||\leq ||Ay||\leq C_2||y||$ for all $y\in Im(A^T)$. Hence, if $y_n\in Im(A^T)$ is any sequence so that $||y_n||\rightarrow\infty$, then $u_n=Ay_n\in Im(AA^T)$ also satisfies $||u_n||\rightarrow\infty$. By computation, we find that
\begin{eqnarray*} 
A(y_n+(Ay_n)^3)=(Ay_n)+A((Ay_n)^3)=u_n+Au_n^3.
\end{eqnarray*}
Therefore, the properness of $F_A(x)$ is equivalent to that of the map $x+Ax^3$ on $Im(AA^T)$. 

\end{proof}

\begin{remarks}We note that the criterion in part 1 of Proposition \ref{MainTheorem} is not sufficient for $F_A$ to be not proper. Here is an example communicated to us by Jiang Liu. Consider the following $5\times 5$ matrix $A$: 
\[ \left( \begin{array}{ccccc}
0&0&1&0&0\\
0&0&0&1&0\\
0&0&0&0&1\\
0&0&0&0&0\\
0&0&0&0&0\\
\end{array}\right) \]
Then $F_A(x)=x+(Ax)^3$ is an automorphism, and hence is proper. On the other hand, with $x=[0,0,1,1,0]$ we have $Ax=[1,1,0,0,0]$ is non-zero, while $A.(Ax)^3=A.[1,1,0,0,0]=0$.

However, see the next subsection where we provide various {\bf algebraic conditions} for the map $F_A$ to be non-proper.   
\label{RemarkMap5Dim}\end{remarks}

{\bf Definition 1.} For clarity of the exposition, we define $\mathcal{P}=\{A:$ $F_A(x)=x+(Ax)^3$ is proper$\}$. 

{\bf Remark.} The following observation, similar to that for the matrices considered in \cite{liu}, is easy to check:  If $A,-A\in \mathcal{P}$, then $\lambda A\in \mathcal{P}$ for all $\lambda \in \mathbb{R}$.

\begin{corollary}
For any given $0\leq r\leq m$, $\mathcal{P}$ contains an open dense set of the set of matrices of rank exactly $r$.   
\label{CorollaryDensity}\end{corollary}
\begin{proof}
By part 1 of Proposition \ref{MainTheorem}, the set of matrices $A\notin \mathcal{P}$ is contained in a projection of the following  variety $\{(A,x):$ $A.(Ax)^3=0$ and $1-||Ax||^2=0\}$. From this, the corollary follows. 
\end{proof}

The proof of Corollary \ref{CorollaryDensity} can be used to prove that $\mathcal{P}$ is dense in other familiar subvarities of the set of all $m\times m$ matrices. However, see Corollary \ref{Corollary3} for a large class of matrices $A$ for which $F_A$ is non-proper. Now we are ready to state and prove the first main result of this paper. 

\begin{theorem}  Let $A$ be an $m\times m$ matrix with real coefficients and $F_A(x)=x+(Ax)^3: ~\mathbb{R}^m\rightarrow \mathbb{R}^m$.  Then $F_A(x)$ is proper if one of the following conditions is satisfied: 

 (1) $Ker(A)\subset Ker(AA^T)$, where $A^T$ is the transpose of $A$. This condition is satisfied for example when $A$ is invertible, symmetric or anti-symmetric.  
 
 (2) $AA^T$ has rank $1$. 
 
 (3) $A$ is an upper or lower triangular matrix.  
 
 (4) $F_B(x)$ is proper, where $B=DAD^{-3}$ for an invertible {\bf diagonal} matrix $D$.  
 
 (5) There is a vector $\zeta =(\zeta _1,\ldots ,\zeta _m)$ with $\zeta _i>0$ for all $i$ such that $<A(x^3),\zeta * x>\geq 0$ for all $x\in Im(AA^T)$. Then $F_A(x)$ is proper. If moreover, $<A(x^3),\zeta * x> \not= 0$ for all $0\not= x\in Im(AA^T)$, then $F_{-A}(x)$ is also proper. 
 
  (6) $Ker(A)$ is generated by the vector $x_{\infty}=(1,\ldots ,1)$ and $x_{\infty}\notin Im(AA^T)\cap Im(AAA^T)$. 
\label{Theorem1}\end{theorem}
\begin{proof}
 
 1) Assume that $Ker(A)\subset Ker(AA^T)$. We will show that $F_A(x)$ is proper. Assume otherwise, then by 1) of Proposition \ref{MainTheorem}, there is $x\in A^T$ such that $A.(Ax)^3=0$ but $Ax\not=0$. We write $x=A^Tu$. Then on the one hand
 \begin{eqnarray*}
<(Ax)^3,Ax>=||(Ax)^2||^2>0,
\end{eqnarray*} 
since $Ax\not= 0$. On the other hand
\begin{eqnarray*}
<(Ax)^3,Ax>=<(Ax)^3,AA^Tu>=<AA^T.(Ax)^3,u>=0,
\end{eqnarray*}
since $(Ax)^3\in Ker(A)$ and by assumption we have $Ker(A)\subset Ker(AA^T)$. Therefore, we obtain a contradiction. Thus $F_A(x)$ is proper. 

2) Assume that $AA^T$ has rank $1$, we will show that $F_A(x)$ is proper. By part 3 of Proposition \ref{MainTheorem}, $F_A(x)$ is proper iff $x+Ax^3$ is proper on $Im(AA^T)$. By assumption, $Im(AA^T)$ has dimension $1$, and hence the map $x+Ax^3$ is obviously proper.

3) We assume that $A$ is an upper triangular matrix, and will show that $F_A$ is proper. Assume otherwise, there is a sequence $x_n$ so that $||x_n||\rightarrow \infty$ and $x_n+(Ax_n)^3$ is bounded. We write $x_{n}=(x_{n,1},\ldots ,x_{n,m})$, and assume that $1\leq j\leq m$ is the largest number so that $\limsup _{n\rightarrow\infty}|x_{n,j}|=\infty$. The $j$-th coordinate of $x_n+(Ax_n)^3$, by the assumption that $A$ is an upper triangular matrix, is: $x_{n,j}+(a_{j,j}x_{n,j}+\sum _{k=j+1}^ma_{j,k}x_{n,k})^3$. By assumption, we have that $\sum _{k=j+1}^ma_{j,k}x_{n,k}$ is bounded, $\limsup _{n\rightarrow\infty}||x_{n,j}||=\infty$ and  $x_{n,j}+(a_{j,j}x_{n,j}+\sum _{k=j+1}^ma_{j,k}x_{n,k})^3$ is bounded. This is a contradiction. 

 4) This is exactly part 2) of Proposition \ref{MainTheorem}. 
 
 5) If $F_A$ is non-proper, then by Proposition \ref{MainTheorem} there is a sequence $x_n\in Im(AA^T)$ for which $||x_n||\rightarrow\infty$ and $x_n+A(x_n^3)$ is bounded. Then $x_n*x_n+A(x_n^3)*x_n$ is $O(x_n)$. On the other hand, by the assumption on $\zeta$:
 \begin{eqnarray*}
 <x_n*x_n+A(x_n^3)*x_n,\zeta >\geq \zeta ^*||x_n||^2,
 \end{eqnarray*}
 where $\zeta ^*=\min _{i=1,\ldots ,m}\zeta _i>0$. Thus we obtain a contradiction. 
 
 In the case $<A(x^3)*x,\zeta > \not= 0$ for all $0\not= x\in Im(AA^T)$, we have $<A(x^3)*x,\zeta >\geq c||x^2||^2$ for some constant $c>0$ and for all $x\in Im(AA^T)$. Hence, in this case, also $x-(Ax)^3$ is proper.  
 
 6) This is a special case of Corollary \ref{Corollary4} in the next subsection, where a sufficient and necessary condition is given. 
 \end{proof}

\begin{remarks}We note that if $A$ is  invertible or symmetric or anti-symmetric, then $Ker(A)\subset Ker(AA^T)$ and hence part 1) of Theorem \ref{Theorem1} can be applied. We note that if $A$ has rank $1$, then $AA^T$ has rank $1$, and hence part 2) of Theorem \ref{Theorem1} can be used.

We note that Theorem \ref{Theorem1} can be extended to the case where $A$ has coefficients in $\mathbb{C}$ and $F_A(x):~\mathbb{C}^m\rightarrow \mathbb{C}^m$, by using appropriate modifications. For example, we replace the transpose by conjugate transpose. Another way is to reduce properness of a map on $\mathbb{C}^m$ to that of a map on $\mathbb{R}^{2m}$. 
\end{remarks}

\subsection{Non-properness of identity plus cubics} In this section, we provide an (almost) necessary and sufficient {\bf linear algebraic condition} for $x+A(x^3)$ to be non-proper, thereby deriving a sufficient linear algebraic condition for $F_A(x)$ to be non-proper. It is interesting that this is possible because of the simple fact that any real number has a unique real third-root. As a byproduct, we provide a large class of matrices $A$ for which $F_A(x)$ is non-proper. We start with an explicit example of such a matrix $A$. 

\begin{remarks}This example is communicated to us by Jiang Liu. There are matrices $A$ for which $F_A$ is non-proper. Indeed, let $f$ be the example by Pinchuk mentioned above, which is a polynomial mapping on $\mathbb{R}^2$ satisfying the following conditions: (i) $Jf$ is invertible everywhere, and (ii) $f$ is not injective. By \cite{druzkowski}, there is a map $F_A(x)=x+(Ax)^3$, in a big dimension $m$, with the same properties (i) and (ii). By Hadamard's theorem, it follows that $F_A(x)$ is not proper. Indeed, if $z$ is the coordinate of $\mathbb{R}^2$, then there is a new variable $y$ (in $m-2$ dimension for some large $m$) so that there are automorphisms of $C,D$ of $\mathbb{R}^{m}$ for which $C\circ (f(z),y)\circ D=F_A(x)$, for $x=(y,z)$. 

We note that the dimension of this example is quite big, and the example is itself quite complicated. By using results presented next, we will construct more explicit examples in the smallest possible dimension $3$. 
\end{remarks}

By Proposition \ref{MainTheorem}, that $F_A(x)=x+(Ax)^3$ is non-proper is equivalent to that there is an unbounded sequence $x_n\in Im(AA^T)$ so that $x_n+A(x_n^3)$ is bounded. We assume that $x_n/||x_n||$ converges to $x_{\infty}$. By Proposition \ref{MainTheorem}, we know that $A.(x_{\infty}^3)=0$, but the above Remark shows that this is not sufficient.  The next result gives a simple necessary and sufficient condition, under the (generic) assumption that all coordinates of $x_{\infty}$  are non-zero, plus an additional technical condition which we feel is not necessarily needed.

{\bf Definition 2.} Let $B\subset \mathbb{R}^m$ be a subset. We denote $B^{1/3}=\{x\in \mathbb{R}^3:~x^3\in B\}$. We note that the map $B^{1/3}\mapsto B$ given by $x\mapsto x^3$ is bijective.

\begin{theorem}
Let $A$ be an $m\times m$ matrix with real coefficients and $\widehat{F}_A(x)=x+A(x^3):~\mathbb{R}^m\rightarrow \mathbb{R}^m$. Let $V\subset \mathbb{R}^m$ be a subvector space containing $Ker(A)^{1/3}$. Let $x_{\infty}\in V$ be a vector whose every coordinates are non-zero. Then the following two conditions are equivalent: 

1)  There is a sequence $x_n\in V$ so that $||x_n||\rightarrow \infty$,  $\widehat{F}_A(x_n)$ is bounded, and $x_n/||x_n||\rightarrow x_{\infty}$. 

2) $x_{\infty}\in A(V*x_{\infty}^{2})\cap Ker(A)^{1/3}$. 

\label{Theorem2}\end{theorem} 
\begin{proof}

Proof of 1) $\Rightarrow $ 2): Assume that $\widehat{F}_A(x_n)$ is bounded. Then by the arguments in the proof of Proposition \ref{MainTheorem}, we have that $x_{\infty}\in Ker(A)$. Since $x_n +A(x_n^3)$ is bounded, there is $u_n\in \mathbb{R}^m$ so that $Au_n=A(x_n^3)$ and $||u_n||\sim ||A(x_n^3)|| \sim ||x_n||$. Hence, we can write
\begin{eqnarray*}
x_n^3=w_n+u_n,
\end{eqnarray*}
where $w_n\in Ker(A)$. Since $Ker(A)^{1/3}\subset V$, we can find $z_n\in V$ so that $z_n^3=w_n$. By assumptions on $x_{\infty}, x_n,w_n$ and $u_n$, it follows that for large enough $n$, all coordinates of $z_n$ are also non-zero. Moreover, 
\begin{eqnarray*}
\lim _{n\rightarrow\infty}||x_n||/||z_n||&=&1,\\
\lim _{n\rightarrow\infty}x_n/||x_n|| &=& \lim _{n\rightarrow\infty}z_n/||x_n|| =x_{\infty}. 
\end{eqnarray*}

Since each real number has a unique real third root, we obtain by Taylor's expansion and with $\gamma _n=||x_n||$
\begin{eqnarray*}
x_n=z_n(1+u_n*z_n^{-3})^{1/3}=z_n+\frac{1}{3}u_n*z_n^{-2}+O(1/\gamma _n^3). 
\end{eqnarray*}

We can assume that $u_n/\gamma _n\rightarrow u\in \mathbb{R}^m$.  Then 
\begin{eqnarray*}
\lim _{n\rightarrow\infty}\gamma _n(x_n-z_n)=\lim _{n\rightarrow\infty}\gamma _n\frac{1}{3}u_n*z_n^{-2}=\frac{1}{3}u*x_{\infty}^{-2}.
\end{eqnarray*}

Because for every $n$, the term $\gamma _n(x_n-z_n)$ is in $V$, it follows that the limit is also in $V$. Therefore $u*x_{\infty}^{-2}\in V$, or equivalently $u\in V*x_{\infty}^2$.  
 
Now, since $Au_n=A(x_n^3)$ for all $n$ and $x_n+A(x_n^3)$ is bounded, it follows that
\begin{eqnarray*}
0=\lim _{n\rightarrow\infty}\frac{1}{\gamma _n} [x_n+Au_n]=x_{\infty}+Au.
\end{eqnarray*}
This implies that $x_{\infty}=A(-u)\in A(V*x_{\infty}^2)$. 

Proof of 2) $\Rightarrow$ 1): Assume that $x_{\infty}\in Ker(A)\cap A(V*x_{\infty}^2)$. Then, there is $u\in V*x_{\infty}^2$ such that $Au+x_{\infty}=0$. We choose a sequence $\gamma _n\rightarrow\infty$, and define 
\begin{eqnarray*}
x_n=\gamma _nx_{\infty}+\frac{1}{3\gamma _n}u*x_{\infty}^{-2}.
\end{eqnarray*}
Since $x_{\infty}\in V$ and $u*x_{\infty}^{-2}\in (V*x_{\infty}^2)*x_{\infty}^{-2}=V$, we have that $x_n\in V$ for all $n$. We have 
\begin{eqnarray*}
x_n^3=\gamma _n^3x_{\infty}^3+\gamma _nu+O(1/\gamma _n). 
\end{eqnarray*}
Because $x_{\infty}\in Ker(A)^{1/3}$, we have $x_{\infty}^3\in Ker (A)$. Therefore, using $x_{\infty}+Au=0$, we obtain
\begin{eqnarray*}
x_n+A(x_n^3)&=&[\gamma _nx_{\infty}+\frac{1}{3\gamma _n}u*x_{\infty}^{-2}]+A(\gamma _n^3x_{\infty}^3+\gamma _nu+O(1/\gamma _n))\\
&=&\gamma _nx_{\infty}+A(\gamma _nu) + O(1/\gamma _n)\\
&=&O(1/\gamma _n). 
\end{eqnarray*} 
Therefore, $x_n+A(x_n^3)$ is bounded. Moreover, $||x_n||\rightarrow\infty$ and $x_n/||x_n||\rightarrow x_{\infty}$.  
\end{proof}

Even if $x_{\infty}$ can have some zero-coordinates or $Ker(A)^{1/3}\not\subset V$, the proof Theorem \ref{Theorem2} can be adapted to say something more. We leave it to the readers for carrying out the suitable modifications. Here we illustrate this idea for the special case where $A$ has corank $1$ where a complete characterisation can be obtained. See also Theorem \ref{TheoremNecessarySufficient} at the end of this subsection for another illustration, where we show that the necessary conditions and sufficient conditions for a map $F_A(x)$ is proper are not too much far apart. (For simplicity of the proof we consider here $x_{\infty}=(1,\ldots, 1,0,\ldots ,0)$, but it is easy to modify the proof to the case where $x_{\infty}$ is an arbitrary vector whose first $k$ coordinates are non-zero and the last $m-k$ coordinates are $0$, for example by transforming  $A$ into $DAD^{-3}$ for some appropriate diagonal matrix $D$.)  

\begin{theorem} Let $A$ be an $m\times m$ matrix of  corank $1$.  Let $x_{\infty}$ be a vector whose first $k$ coordinates are non-zero and whose last $m-k$ coordinates are $0$, where $k\geq 1$.  Let $V\subset \mathbb{R}^m$ be a vector subspace containing $x_{\infty}$, and $pr:\mathbb{R}^m\rightarrow \mathbb{R}^k$ is the usual projection to the first $k$ coordinates. Denote also by $\widehat{x}=x-pr(x)$. 

Then the following 2 statements are equivalent. 

1) There exists a sequence $x_n\in V$ with $||x_n||\rightarrow \infty$, $x_n/||x_n||\rightarrow x_{\infty}$ and $x_n+A(x_n^3)$ is bounded.

2) Some algebraic conditions must be satisfied. The first condition is that there is $u\in \mathbb{R}^m$ such that $Au+x_{\infty}=0$. 

In case all coordinates of $\widehat{u}$, considered as an element of the vector space $Ker(pr)=\mathbb{R}^{m-k}$, are non-zero,  then these algebraic conditions are: There exist $u,v\in \mathbb{R}^m$ such that

\begin{eqnarray*}
x_{\infty}^3&\in&Ker(A),\\ 
Au+x_{\infty}&=&0,\\
Av+\widehat{u}^{1/3}&=&0,\\
\widehat{u}^{1/3}, \widehat{v}*\widehat{u}^{-2/3}&\in& V,\\
pr(u),pr(v)&\in& pr(V).
\end{eqnarray*}

In the general case, we have extra conditions which can be defined inductively, which are too complicated to explain here, see the proof of the theorem for more detail. 

\label{Theorem2Bis}\end{theorem}
\begin{proof}
We note that here $pr(x)=x*(1,\ldots, 1,0,\ldots ,0)$. Recall that $\widehat{x}=x-pr(x)$, and that $Im(pr)=\mathbb{R}^k$.  

{\bf Proof of 1) $\Rightarrow $ 2)}:

Assume that there exists a sequence $x_n\in V$ with $||x_n||\rightarrow \infty$, $x_n/||x_n||\rightarrow x_{\infty}$ and $x_n+A(x_n^3)$ is bounded. Then it follows by the proof of Proposition \ref{MainTheorem} that $x_{\infty}^3\in Ker(A)$. Then, running the part 1) $\Rightarrow $ 2) of  the proof of Theorem \ref{Theorem2}, we can find $u\in \mathbb{R}^m$ such that $Au+x_{\infty}=0$ and $pr(u)\in V*x_{\infty}^2$. Hence, $u\in pr^{-1}(V*x_{\infty}^2)$. We will later show that indeed the stronger property $pr(u)\in V$ holds. 

If $\widehat{u}=0$ then we stop. Below, we assume that $\widehat{u}\not= 0$. 

\underline{Case 1: All coordinates (of course, from $k+1$ to $m$) of $\widehat{u}$ are non-zero.}

Since in this case the $Ker(A)$ is generated by $x_{\infty}$, we can write, for $||x_n||\sim \gamma _n$
\begin{eqnarray*}
x_n^3=\gamma _nx_{\infty}^3+u_n,
\end{eqnarray*}
where $||u_n||\sim \gamma$. Here $u_n/\gamma _n\rightarrow u$. What we described in the previous paragraph can be translated into (by carefully look at coordinates $1\leq j\leq k$ and $k+1\leq j\leq $ separately):
\begin{eqnarray*}
x_n=\gamma _nx_{\infty}+\frac{1}{3\gamma _n^2}pr(u_n)+\widehat{u}_n^{1/3}+O(1/\gamma _n^2). 
\end{eqnarray*}

Then the condition that $x_n+A(x_n^3)=O(1)$ implies
\begin{eqnarray*}
\gamma _nx_{\infty}+Au_n+ \widehat{u}_n^{1/3}=O(1).
\end{eqnarray*}

If we write $u_n=\gamma _nu+v_n$, then $v_n/\gamma _n\rightarrow 0$. Using that $x_{\infty}=-Au$, we get
\begin{eqnarray*}
Av_n+(\gamma _n\widehat{u}+\widehat{v}_n)^{1/3}=O(1). 
\end{eqnarray*}

Writing $Av_n=A.pr(v_n)+A\widehat{v}_n$, and note that $A.pr(v_n)\in A(\mathbb{R}^k)$, we have that $A\widehat{v_n}+(\gamma _n\widehat{u}+\widehat{v}_n)^{1/3}+O(1)\in A(\mathbb{R}^k)$. 

Now, if $\lim _{n\rightarrow\infty}||\widehat{v_n}||/\gamma _{n}^{1/3}=0$, then we have 
\begin{eqnarray*}
\widehat{u}^{1/3}=\lim _{n\rightarrow\infty}\frac{1}{\gamma _n^{1/3}}[A\widehat{v_n}+(\gamma _n\widehat{u}+\widehat{v}_n)^{1/3}+O(1)]\in A(\mathbb{R}^k)
 \end{eqnarray*}

If $\lim _{n\rightarrow\infty}||\widehat{v_n}||/\gamma _{n}^{1/3}=c$ a non-zero finite limit, then putting $\widehat{v}=\lim _{n\rightarrow\infty}v_n/\gamma _n^{1/3}$, we get  by taking a limit as before $A\widehat{v}+\widehat{u}^{1/3}\in A(\mathbb{R}^k)$.  Hence, again $\widehat{u}^{1/3}\in Im(A)$.

If $\lim _{n\rightarrow\infty}||\widehat{v_n}||/\gamma _{n}^{1/3}=\infty$, then by taking a limit as before we obtain that $A\widehat{v}\in A(\mathbb{R}^k)$ where $\widehat{v}=\lim _{n\rightarrow\infty}v_n/||v_n||$. However, this is a contradiction with the fact that $pr(\widehat{v} )=0$, $||\widehat{v}||=1$ and $Ker(A)=<x_{\infty}>\subset \mathbb{R}^k$. 

From the above analysis, we can assume that $\lim _{n\rightarrow\infty}v_n/\gamma _n^{1/3}=v$ exists. Using that $(\gamma _n\widehat{u}+\widehat{v_n})^{1/3}=\gamma _n^{1/3}\widehat{u}+O(1)$, we obtain that $Av_n+\gamma _n^{1/3}\widehat{u}=O(1)$. Dividing this by $\gamma _n^{1/3}$ and taking the limit when $n\rightarrow\infty$, we obtain $Av+\widehat{u}^{1/3}=0$. 

Now, we can write $v_n=\gamma _n^{1/3}v+w_n$, where $\lim _{n\rightarrow\infty}w_n/\gamma _n^{1/3}=0$. Then we deduce that $Aw_n=O(1)$. This implies that if we write $w_n=\lambda _nx_{\infty}+w_n^*$, then $w_n^*$ is bounded (recall that $Ker(A)$ is generated by $x_{\infty}$ and hence for every $x$ so that $<x,x_{\infty}>=0$ then $||Ax||\geq c||x||$ for some positive constant $c$). We can therefore, assume that $w_n^*$ converges to some $w^*$. 

Using Taylor's expansion 
\begin{eqnarray*}
(1+t)^{1/3}=1+\frac{1}{3}t-\frac{2}{9}t^2+\frac{10}{27}t^3+O(t^4)
\end{eqnarray*}
for $t$ small enough, we obtain
\begin{eqnarray*}
(\gamma _n\widehat{u}+\widehat{v_n})^{1/3}&=&(\gamma _n\widehat{u}+\gamma _n^{1/3}\widehat{v}+\widehat{w_n^*})^{1/3}\\
&=&\gamma _n^{1/3}\widehat{u}+\frac{1}{3\gamma _n^{2/3}}(\gamma _n^{1/3}\widehat{v}+\widehat{w_n^*})*\widehat{u}^{-2/3}-\frac{2}{9\gamma ^{5/3}}(\gamma _n^{1/3}\widehat{v}+\widehat{w_n^*})^2*\widehat{u}^{-5/3}\\
&&+\frac{10}{27\gamma _n^{8/3}}  (\gamma _n^{1/3}\widehat{v}+\widehat{w_n^*})^3*\widehat{u}^{-8/3}+O(1/\gamma _n^{7/3}). 
\end{eqnarray*}

Therefore, we can write
\begin{eqnarray*}
x_n&=&\gamma _nx_{\infty}+\frac{1}{3\gamma _n^2}pr(\gamma _nu+\gamma _n^{1/3}v+\lambda _nx_{\infty})+(\gamma _n\widehat{u}+\gamma ^{1/3}\widehat{v}+\widehat{w_n^*})^{1/3}+O(1/\gamma _n^2). 
\end{eqnarray*}
Since $x_n,x_{\infty}\in V$ for all $n$, by looking at terms of orders $\gamma _n^{1/3}$ and $\gamma _n^{-1/3}$ of the RHS in the above expressions (which correspond to terms of second and third biggest sizes) , we obtain that $\widehat{u}^{1/3}, \widehat{v}*\widehat{u}^{-2/3}\in V$.  

On the other hand, applying the projection map $pr$, we obtain 
\begin{eqnarray*}
pr(x_n)=\gamma _npr(x_{\infty})+\frac{1}{3\gamma _n^2}pr(\gamma _nu+\gamma _n^{1/3}v+\lambda _nx_{\infty})+O(1/\gamma _n^2)
\end{eqnarray*}
for all $n$. It follows as above that $pr(u), pr(v)\in pr(V)$. Hence, we obtain all the algebraic conditions stated in 2).

\underline{Case 2: Some coordinates (of course, from $k+1$ to $m$) of $\widehat{u}$ are zero.}

Now we explain how to get extra algebraic conditions in the case $\widehat{u}$ has some, but not all, zero coordinates. Let $\tau $ be the projection from $Ker(pr)=\mathbb{R}^{m-k}$ onto the vector subspace $V_1$ for which coordinates of $\widehat{u}$ are $0$. From the above analysis, we see that $\limsup _{n\rightarrow\infty}||v_n||/\gamma _n^{1/3}$ is a finite number. In case that number is $0$, then we have $\widehat{u}^{1/3}\in A(\mathbb{R}^k)$ and stop. Otherwise, we note that 
\begin{eqnarray*}
(\gamma _n\widehat{u}+\widehat{v}_n)^{1/3}=\gamma _n^{1/3}\widehat{u}^{1/3}+\tau (\widehat{v}_n)^{1/3}+O(1/\gamma _n^{2/3}). 
\end{eqnarray*}

Therefore, the condition we considered before now become
\begin{eqnarray*}
Av_n+\gamma _n^{1/3}\widehat{u}^{1/3}+\tau (\widehat{v}_n)^{1/3}=O(1). 
\end{eqnarray*}
We use the trick as before, writing $v_n=\gamma _n^{1/3}v+w_n$, where $w_n/\gamma _n^{1/3}\rightarrow 0$. Then the condition above becomes: 
\begin{eqnarray*}
A(\gamma _n^{1/3}v+w_n)+\gamma _n^{1/3}\widehat{u}^{1/3}+ (\gamma _n^{1/3}\tau \widehat{v}+\tau (\widehat{w_n}))^{1/3}=O(1). 
\end{eqnarray*}

Since $Av+\widehat{u}^{1/3}=0$, the above reduces to $Aw_n+(\gamma _n^{1/3}\tau \widehat{v}+\tau (\widehat{w_n}))^{1/3}=O(1)$. This is exactly the case which we dealt with before, where $v_n$ was in the place of $w_n$ and $u_n$ was in the place of $v_n$. We then proceed as before to obtain an addition condition: $\tau \widehat{v}^{1/3}\in Im(A)$. Then, as before, depending on whether $\tau \widehat{v}$, considered as an element of $V_1$, have zero coordinates or not, we will proceed similarly or stop. Since there are at most $m-k$ coordinates to consider, we will stop after at most $m-k$ such considerations. At each step, we add some more algebraic conditions. Thus, at the end we obtain a finite number of algebraic equations which $x_{\infty}$ must satisfy if there is such a sequence $x_n$ as in the statement of the theorem.

{\bf Proof of 2) $\Rightarrow $ 1):} 

Here we give the proof for the case where all coordinates from $k+1$ to $m$ of $\widehat{u}$ are non-zero, the other cases can be treated similarly. In this case, the algebraic conditions in part 2 are: There exist $u,v\in \mathbb{R}^m$ with the following properties:
\begin{eqnarray*}
x_{\infty}^3&\in& Ker(A),\\
Au+x_{\infty}&=&0,\\
Av+\widehat{u}^{1/3}&=&0,\\
\widehat{u}^{1/3}, \widehat{v}*\widehat{u}^{-2/3}&\in& V,\\
pr(u),pr(v)&\in& pr(V).
\end{eqnarray*}

We thus can find $u_1,v_1\in V$ so that $pr(u)=pr(u_1)$ and $pr(v)=pr(v_1)$. We then choose a sequence $\gamma _n\rightarrow\infty$ and define
\begin{eqnarray*}
x_n=\gamma _nx_{\infty}+\frac{1}{3\gamma _n^2}(\gamma _nu_1+\gamma _n^{1/3}v_1)+\gamma _n^{1/3}\widehat{u}^{1/3}+\frac{1}{3\gamma _n^{1/3}}\widehat{v}*\widehat{u}^{-2/3}.
\end{eqnarray*}

From the properties of $u,v,u_1,v_1$ we have that $x_n\in V$ for all $n$. Also, it is easy to check that $||x_n||\rightarrow \infty$ and $x_n/||x_n||\rightarrow x_{\infty}$.  Moreover, since $pr(u)=pr(u_1)$ and $pr(v)=pr(v_1)$
\begin{eqnarray*}
x_n^3&=&\gamma _n^3x_{\infty}^3+\gamma _npr(u_1)+\gamma _n^{1/3}pr(v_1)+\gamma _n\widehat{u}+ \gamma _n^{1/3}\widehat{v}+O(1/\gamma _n^{1/3})\\
&=&\gamma _n^3x_{\infty}^3+\gamma _n u+\gamma _n^{1/3}v+O(1/\gamma _n^{1/3}). 
\end{eqnarray*}
Therefore, $x_n+A(x_n^3)=O(1/\gamma _n^{1/3})$, as wanted.   

\end{proof} 

\begin{remarks}
In the proof of Case 1 in part 1 of Theorem \ref{Theorem2}, we can choose from beginning that $u_n\in Im(A^T)$. Then $v_n$ is also in $Im(A^T)$, and from this the case $\lim _{n\rightarrow\infty}||v_n||/\gamma _n^{1/3}=\infty$ can be quickly discarded, using that $||Ax||\geq c||x||$ for all $x\in Im(A^T)$ for some positive constant $c$, without the need that $x_{\infty}$ generates $Ker(A)$. 

\end{remarks}

As a corollary, we obtain a large class of matrices for which $F_A(x)$ is non-proper. 

\begin{corollary} Assume that there is a vector $x_{\infty}$ whose all coordinates are non-zero with the property that $x_{\infty}\in Im(AA^T)\cap Ker(A)^{1/3}\cap A(Im(AA^T)*x_{\infty}^2)$. Then $F_A(x)=x+(Ax)^3$ is non-proper. 
\label{Corollary3}\end{corollary} 
\begin{proof}
 By Proposition \ref{MainTheorem}, to show the non-properness of $F_A(x)$ it suffices to show the existence of a sequence $x_n\in Im(AA^T)$ so that $||x_n||\rightarrow \infty$ and $x_n+A(x_n^3)$ is unbounded. We can then apply Theorem \ref{Theorem2Bis} for $V=Im(AA^T)$.
\end{proof}  
  
\begin{remarks} If $x_{\infty}\in Im(AA^T)$, then it is obvious that there is $u\in Im(A^T)$ so that $Au=-x_{\infty}$. However, it is crucial for the arguments used here that $u\in Im(AA^T)*x_{\infty}^2$.   
\end{remarks}
  
In the opposite direction, we provide another large class of matrices of corank $1$ for which $F_A$ is proper. The proof is a simple application of Theorem \ref{Theorem2} and hence is skipped. 
\begin{corollary}
Let $A$ be an $m\times m$ matrix of corank $1$. Assume that $Ker(A)$ is generated by the vector $x_{\infty}=(1,\ldots ,\ldots ,1)$ (all coordinates are $1$). Then $F_A(x)$ is proper iff $x_{\infty}\notin Im(AA^T)\cap A(Im(AA^T))$. 

More generally, if $Ker(A)$ is generated by the vector $x_{\infty}=(1,\ldots ,1,0,\ldots ,0)$ (the first $k$ coordinates are $1$ and the last $m-k$ coordinates are $0$), then $F_A(x)$ is proper iff the algebraic conditions in part 2) of Theorem \ref{Theorem2Bis} are not satisfied. 

\label{Corollary4}\end{corollary}
\begin{proof}
Note that if $F_A(x)$ is non-proper, then we must have $x_{\infty}\in Im(AA^T)=:V$.  Hence, we can assume from the beginning that $x_{\infty}\in V$, and hence also $Ker(A)^{1/3}=Ker(A)\subset V$. Then the corollary follows from Theorem \ref{Theorem2Bis}. 
\end{proof}

\begin{remark}

We now use Corollary \ref{Corollary4} to construct examples of $3\times 3$ matrices $A$ for which $F_A(x)$ is not proper. We write such a matrix as:
\[ \left( \begin{array}{ccc}
a_{1,1}&a_{1,2}&a_{1,3}\\
a_{2,1}&a_{2,2}&a_{2,3}\\
a_{3,1}&a_{3,2}&a_{3,3}\\
\end{array}\right) \]
We want $A$ to be a matrix of rank $2$, whose kernel $Ker(A)$ is generated by the vector $x_{\infty}=(1,1,1)$ and $x_{\infty}\in Im(AA^T)\cap Im(A^2A^T)$. Hence, after permuting of coordinates, from the assumption that $A$ is of rank $2$ we get 
\begin{eqnarray*}
(a_{3,1},a_{3,2},a_{3,3})=\lambda (a_{1,1},a_{1,2},a_{1,3})+\mu (a_{2,1},a_{2,2},a_{2,3}),
\end{eqnarray*}
for some real numbers $\lambda $ and $\mu$. 

The condition that $x_{\infty}\in Ker(A)$ is generated into  two equations $a_{1,1}+a_{1,2}+a_{1,3}=0$ and $a_{2,1}+a_{2,2}+a_{2,3}=0$. 

Since $Im(AA^T)=Im(A)$ (recall that $\mathbb{R}^m=Im(A^T)\oplus Ker(A)$), we have that $Im(AA^T)=\{(x,y,z):~z=\lambda x+\mu y\}$. Hence, the condition that $x_{\infty}\in Im(AA^T)$ is generated into that $1=\lambda +\mu$. 

It then follows that $Im(AA^T)=Im(A)$ is generated by $x_{\infty}$ and $(1,0,\lambda)$. Since $Im(A^2A^T)=Im(A^2)=A(Im(A))$ and $Ax_{\infty}=0$,  the condition that $x_{\infty}\in Im(A^2)$ is generated into that: there is a real number $c$ for which $A.(c,0,\lambda c)=x_{\infty}$. This is translated into that $a_{1,1}+a_{1,3}\lambda =a_{2,1}+a_{2,3}\lambda \not= 0$. 

We note that there are 4  constraints between the parameters $a_{1,1},a_{1,2},a_{1,3},a_{2,1},a_{2,2},a_{2,3},\lambda ,\mu$ (and an inequality $a_{2,1}+a_{2,3}\lambda \not= 0$), which give rise to a constructible set of dimension $4$.

\end{remark}

Theorem \ref{Theorem2} suggest us to propose the following question: 

{\bf Conjecture C.} The set $\mathcal{P}$ (see Definition 1) is constructible. Moreover, we conjecture that it is defined by the algebraic conditions given in part 2 of Theorem \ref{Theorem2Bis}. 

Theorem \ref{Theorem2Bis} gives an affirmative answer to Conjecture in case $A$ has corank $1$. The results in the previous subsection provide an affirmative answer to this conjecture to the cases where $A$ is invertible or has rank $1$, as well as several other cases. Also, Hadamard's theorem and the fact that the set of matrices for which $F_A(x)$ is injective provide an affirmative answer to Conjecture C for injective maps $F_A(x)$. We will go more deeply into this case in the next subsection. Now we will give more support for Conjecture C by showing that the necessary conditions and sufficient conditions for a map $F_A(x)$ to be proper are not too far apart. 

{\bf Definition 3.} Let $A$ be an $m\times m$ matrix, $0\not= x_{\infty}\in \mathbb{R}^m$ and $V\subset \mathbb{R}^m$ a vector subspace. We define $\mathcal{S}(x_{\infty},V)$ to be the set of algebraic conditions given in part 2) of Theorem \ref{Theorem2Bis}. For example: 

If all coordinates of $x_{\infty}$ are non-zero, then $\mathcal{S}(x_{\infty},V)$ consists of two conditions (see also Theorem \ref{Theorem2}):  $x_{\infty}^3\in Ker(A)$ and there exists $u\in V*x_{\infty}^2$ so that $Au+x_{\infty}=0$. (Note that here we disregard the assumptions in Theorem \ref{Theorem2} that $Ker(A)^{1/3}\subset V$.)

If the first $k$ coordinates of $x_{\infty}$ are non-zero and the last $m-k$ coordinates of $x_{\infty}$ are non-zero, then we define $pr:\mathbb{R}^m\rightarrow \mathbb{R}^{k}$ be the projection onto the first $k$ coordinates. Then $\mathcal{S}(x_{\infty},V)$ contains the condition that there is $u$ so that $Au+x_{\infty}=0$. Depending on where the coordinates of $\widehat{u}$, considered as an element of $Ker(pr)$, are $0$, we have different cases for defining $\mathcal{S}(x_{\infty},V)$. In the simplest case, where all $m-k$ coordinates of $\widehat{u}$ are non-zero, then $\mathcal{S}(x_{\infty},V)$ consists of the following extra conditions: 

\begin{eqnarray*}
x_{\infty}^3&\in&Ker(A),\\
Au+x_{\infty}&=&0,\\
Av+\widehat{u}^{1/3}&=&0,\\
\widehat{u}^{1/3}, \widehat{v}*\widehat{u}^{-2/3}&\in& V,\\
pr(u),pr(v)&\in& pr(V).
\end{eqnarray*}

(Again, note that we disregard the assumption that $A$ has corank $1$ in Theorem \ref{Theorem2Bis}.)

{\bf Definition 4.} Let $A$ be an $m\times m$ matrix and $0\not= x_{\infty}\in \mathbb{R}^m$. We define $\mathcal{N}(x_{\infty})$ to be the set of algebraic conditions obtained in the proof of the step of 1) $\Rightarrow$ 2 in Theorem \ref{Theorem2Bis}, but we disregard the assumptions that $A$ has corank $1$ or $Ker(A)^{1/3}$. In particular, we have to modify the first part by writing $x_n^3=z_n+u_n$, for some $z_n\in Ker(A)$. We note that $\mathcal{N}(x_{\infty})$ is a subset of $\mathcal{S}(x_{\infty},V)$, where conditions concerning images of $A$ are kept, while conditions  having the vector space $V$ in the statement are discarded. For example: 

 If all coordinates of $x_{\infty}$ are non-zero, then $\mathcal{N}(x_{\infty})$ consists of two conditions (see also Theorem \ref{Theorem2}):  $x_{\infty}^3\in Ker(A)$ and there exists $u\in \mathbb{R}^m$ so that $Au+x_{\infty}=0$. 

If the first $k$ coordinates of $x_{\infty}$ are non-zero and the last $m-k$ coordinates of $x_{\infty}$ are non-zero, then we define $pr:\mathbb{R}^m\rightarrow \mathbb{R}^{k}$ be the projection onto the first $k$ coordinates. Then $\mathcal{N}(x_{\infty})$ contains the condition that there is $u$ so that $Au+x_{\infty}=0$. Depending on where the coordinates of $\widehat{u}$, considered as an element of $Ker(pr)$, are $0$, we have different cases for defining $\mathcal{N}(x_{\infty})$. In the simplest case, where all $m-k$ coordinates of $\widehat{u}$ are non-zero, then $\mathcal{N}(x_{\infty})$ consists of the following extra conditions: 

\begin{eqnarray*}
x_{\infty}^3&\in& Ker(A),\\
Au+x_{\infty}&=&0,\\
Av+\widehat{u}^{1/3}&=&0. 
\end{eqnarray*}

Then, the following result is tautological. 

\begin{theorem} Let $A$ be a matrix and $x_{\infty}\in \mathbb{R}^m$. Let $V\subset \mathbb{R}^m$ be a vector subspace.

1) Assume that there is a sequence $x_n\in V$ so that $||x_n||\rightarrow \infty$, $x_n/||x_n||\rightarrow x_{\infty}$, and $x_n+A(x_n^3)$ is bounded. Then $x_{\infty}\in V$ and $\mathcal{N}(x_{\infty})$ must be satisfied. 

2) If $x_{\infty}\in V$ and $\mathcal{N}(x_{\infty})$ is satisfied, then there is a sequence  $x_n\in \mathbb{R}^m$ so that $||x_n||\rightarrow \infty$, $x_n/||x_n||\rightarrow x_{\infty}$, and $x_n+A(x_n^3)$ is bounded. (Here, we may not have $x_n\in V$.)

3) If $x_{\infty}\in V$ and $\mathcal{S}(x_{\infty},V)$ is satisfied, then there is a sequence $x_n\in V$ so that $||x_n||\rightarrow \infty$, $x_n/||x_n||\rightarrow x_{\infty}$, and $x_n+A(x_n^3)$ is bounded.
\label{TheoremNecessarySufficient}\end{theorem}

\subsection{Applications to the (real) Jacobian conjecture}
In this subsection, we present some applications to the  (real) Jacobian conjecture. We are also interested in its real version for the map $F_A$: if the Jacobian of $F_A(x)$ is nowhere zero on $\mathbb{R}^m$, then $F_A(x)$ is invertible. We recall that by Hadamard's theorem, $F_A(x)$ is invertible if and only if it is proper, therefore results from the previous subsections can be applied. The $1$-dimensional example $F_A(t)=1+t^3$ shows that there are matrices $A$ which are not Druzkowski but still $JF_A(x)$ is nowhere zero.

We first provide a new general class of Druzkowsi maps for which the Jacobian conjecture is known. (For some other general classes, see for example \cite{chamberland-meisters, druzkowski, gorni-gasinska-zampieri,  bondt, truong} and references therein. For example,  if $JF_A(x)$ and $JF_A(y)$ commute for all $x,y$ (this is the case if $JF_A(x)$ is symmetric or anti-symmetric), then the Jacobian conjecture holds.)
\begin{proposition}
Let $A=(a_{i,j})_{1\leq i,j\leq m}$ be a Druzkowski matrix with real coefficients. Assume that there are $\delta _1,\ldots ,\delta _m\in \{\pm 1\} $ so that for all $(i,j),(k,l)$: $\delta _i\delta _j\delta _k\delta _l a_{i,j}a_{k,l}\geq 0$. Then the map $F_A(x)=x+(Ax)^3$ satisfies the Jacobian conjecture.
\label{Theorem0}\end{proposition}
\begin{proof}
Let $D$ be the diagonal matrix, whose diagonal entries are $\delta _j$'s. We note that $D^2=Id$. The assumption of the theorem implies that $\widehat{A}=DAD^{-3}$ is a Druzkowski matrix whose entries have the same sign. In the case all entries of $\widehat{A}$ are non-positive, the assertion follows from the result by \cite{yu, chau-nga} mentioned above. Now consider the case all entries of $\widehat{A}$ are non-negative. We let $D_1$ be the diagonal matrix whose all entries are $\sqrt{-1}$. Then $D_1\widehat{A}D_1^{-3}$ is also a Druzkowski matrix whose all entries are non-positive, and again we can apply the result of \cite{yu, chau-nga}.  
\end{proof} 

Applying Theorem \ref{Theorem1} we obtain the following result.
\begin{theorem}
Let $A$ be a matrix so that the Jacobian of $F_A(x)=x+(Ax)^3$ is nowhere zero. Then, $F_A(x)$ is invertible if one of the following conditions is satisfied: 

 (1) $Ker(A)\subset Ker(AA^T)$, where $A^T$ is the transpose of $A$. This condition is satisfied for example when $A$ is invertible, symmetric or anti-symmetric.  
 
 (2) $AA^T$ has rank $1$. 
 
 (3) $A$ is an upper or lower triangular matrix.  
 
 (4) $F_B(x)$ is proper, where $B=DAD^{-3}$ for an invertible {\bf diagonal} matrix $D$.  
 
 (5) There is a vector $\zeta =(\zeta _1,\ldots ,\zeta _m)$ with $\zeta _i>0$ for all $i$ such that $<A(x^3),\zeta * x>\geq 0$ for all $x\in Im(AA^T)$. Then $F_A(x)$ is proper. If moreover, $<A(x^3),\zeta * x> \not= 0$ for all $0\not= x\in Im(AA^T)$, then $F_{-A}(x)$ is also proper. 
 
  (6) $Ker(A)$ is generated by the vector $x_{\infty}=(1,\ldots ,1)$ and $x_{\infty}\notin Im(AA^T)\cap Im(AAA^T)$. 
\label{TheoremJacobian1}\end{theorem}

Applying Theorem \ref{Theorem2Bis} we obtain the following result. 

\begin{theorem}
Let $A$ be an $m\times m$ matrix of  corank $1$.  Let $x_{\infty}$ be a vector whose first $k$ coordinates are non-zero and whose last $m-k$ coordinates are $0$, where $k\geq 1$.  Let $V=Im(A)$, and assume that $x_{\infty}\in V$. Define $pr:\mathbb{R}^m\rightarrow \mathbb{R}^k$ to be the usual projection to the first $k$ coordinates. Denote also by $\widehat{x}=x-pr(x)$. 

Then the following 2 statements are equivalent. 

1) The map $F_A(x)$ is not invertible.

2) Some algebraic conditions must be satisfied, as in the statement and proof of Theorem \ref{Theorem2Bis}. 
\label{TheoremJacobian2}\end{theorem}

\begin{remark}
We note that the algebraic conditions in 2) can be readily checked with specific examples. For example, the readers can try to work out with the example in Remark \ref{RemarkMap5Dim}..

\end{remark}

\begin{theorem} Let $A$ be a matrix. Let $V=Im(A)$. Assume that the Jacobian of $F_A(x)=x+(Ax)^3$ is nowhere zero. 

1) If $F_A(x)$ is not invertible, then $\mathcal{N}(x_{\infty})$ must be satisfied for some $x_{\infty}\in V$. 

2) If $\mathcal{S}(x_{\infty},V)$ is satisfied for some $x_{\infty}\in V$, then $F_A(x)$ is not invertible.
\label{TheoremNecessarySufficient}\end{theorem}

\subsection{Properness of Identity plus other linear powers}  The ideas in the previous subsection can be applied to similar maps $F_{A,k}(x)=x\pm (Ax)^k$, where $k$ is a positive integer. 

We first look at the case $k=1$. In this case, the map becomes $F_{A,1}=(Id+A)x$. Hence, in this case we have a complete answer: $F_{A,1}$ is proper iff $Id+A$ is invertible. 

The results in the previous subsections can be extended to all maps $F_{A,2k+1}$ for $2k+1\geq 3$. 

Maps $F_{A,2k}$ can be treated similarly, we only need to take care of the fact that there can be two real $2k$-th root of a non-negative real number, and no real $2k$-th root of a negative real number. 

We note that if $x_{\infty}=(1,\ldots ,1)$, the vector $u$ and the matrix $A$ satisfies the assumptions of Theorem \ref{Theorem2}, then for all $k$ the proof of Theorem \ref{Theorem2} shows that the sequence 
\begin{eqnarray*}
x_n^{(k)}=\gamma _nx_{\infty}+\frac{1}{k\gamma _n}u,
\end{eqnarray*}
where $\gamma _n\rightarrow\infty$, satisfies: $x_n\in V$, $||x_n||\rightarrow \infty$, and $x_n+(Ax_n)^k$ is bounded. Hence, under this assumption, all maps $x+(Ax)^k$ are non-proper. Thus, for all $k$, there is a matrix $A$ so that $F_{A,k}(x)=x+(Ax)^k$ is non-proper.

\subsection{Conclusions and Acknowledgements} In this paper, inspired by \cite{liu} - where a proof of the Jacobian conjecture via properness of Druzkowski maps $x+(Ax)^3$ is asserted, we provide various necessary conditions and sufficient conditions concerning the properness of maps of the form $x+(Ax)^{k}$ on $\mathbb{R}^m$, in particular when $k=3$ and when $k$ is an even number. We find that for most of matrices $A$, the map $x+(Ax)^3$ is proper. A complete characterisation of the properness, in  the case where $A$ has corank $1$, is obtained. Extending this, we state Conjecture C, which asserts that properness of the map $x+(Ax)^3$ is equivalent to the non-existence of non-zero solutions of a specific system of polynomial equations of degree at most $3$. We provide some applications to the (real) Jacobian conjecture. Because of a result by Druzkowski \cite{druzkowski}, which states that all self-polynomials on $\mathbb{C}^m$ or $\mathbb{R}^m$ can be reduced to a map of the form considered in this paper, our results can be applied to all polynomial self-maps with real or complex coefficients.  We hope that this paper can stimulate the readers to check in detail the arguments in \cite{liu}, and to demonstrate that checking properness of Druzkowski maps on $\mathbb{R}^m$ is a viable approach toward solving the Jacobian conjecture.  
 
 The author would like to thank Jiang Liu for some explanations on his preprint \cite{liu}, as well as for his comments mentioned in the above. The author also would like to thank John Christian Ottem for his comments. This work is supported by Young Research Talents grant 300814 from Research Council of Norway.

\end{document}